\documentclass[letterpaper,12pt,twoside]{amsart}
\usepackage[left=2.5cm, right=2.5cm,tmargin = 2.5cm]{geometry}
\usepackage{amsmath,amssymb,amsfonts}
\usepackage{indentfirst}
\usepackage[utf8]{inputenc}
\usepackage{hyperref}
\usepackage{lastpage}
\usepackage{enumitem}
\usepackage{lineno}

\renewcommand{\title}[1]{\begin{center}
    \begin{minipage}[t]{125mm}
        \Large\bf\begin{center} #1
        \end{center}
    \end{minipage}
\end{center}
\vskip3mm }

\renewcommand{\author}[2]{\centerline{#1}
\vskip2mm
\par
\centerline{\small\it
 \begin{minipage}[t]{145mm}
  \begin{center}
 #2
  \end{center}
 \end{minipage}
}
\vspace{8mm}
}

\newcommand{\resumen}[1]{
\begin{center}
\textbf{Abstract}\\

\begin{minipage}[t]{130mm}
\small
#1
\end{minipage}
\end{center}
}
\newcommand{\palabrasclave}[1]{
\begin{center}
\begin{minipage}[t]{130mm}
\small
\emph{Keywords: }
#1
\end{minipage}
\end{center}
}

\newtheorem{teorema}{Theorem}[section]
\newtheorem{lema}{Lemma}[section]

\newtheorem{proposicion}{Proposition}[section]

\theoremstyle{definition}

\newtheorem{ejemplo}{Example}[section]

\theoremstyle{remark}
\newtheorem{obs}{Observation}[section]

\newcommand\R{\mathbb{R}}

\usepackage{mathrsfs}
\usepackage{graphicx}
\usepackage{color}

\begin{document}
\setcounter{page}{1}

\normalsize

\title{Lie algebra rank condition for bilinear control systems on $\mathbb{R}^2$}

\author{Efrain Cruz-Mullisaca, Victor H. Patty-Yujra}%
{
Instituto de investigación Matemática, Universidad Mayor de San Andrés, Bolivia
}

\newcommand{\Addresses}{{
  \bigskip
  \footnotesize

\begin{itemize}[label={},itemindent=-2em,leftmargin=2em]
\item \textsc{ E.Cruz-Mullisaca}, Instituto de Investigación Matemática, Universidad Mayor de San Andrés.\\  
    Calle 27 de Cota Cota, Campus Universitario, Edificio de Ciencias Puras y Naturales, 1er Piso. La Paz-- Bolivia\par\nopagebreak
    \textit{E-mail:}\texttt{ecruz@fcpn.edu.bo}
\item \textsc{ V.H.Patty-Yujra}, Instituto de Investigación Matemática, Universidad Mayor de San Andrés.\\  
    Calle 27 de Cota Cota, Campus Universitario, Edificio de Ciencias Puras y Naturales, 1er Piso. La Paz-- Bolivia\par\nopagebreak
    \textit{E-mail:}\texttt{vpattyy@fcpn.edu.bo}
\end{itemize}
}}

\hrule
\resumen{
We will study the controllability problem of a bilinear control system on $\mathbb{R}^2:$ the main result is the characterization of the Lie algebra rank condition for the system. On the other hand, using elementary techniques, we recover conditions for the controllability of the induced angular system on the projective space. Finally, we will give controllability criteria for the system.   
}

\palabrasclave{
Bilinear control system; Lie algebra rank condition; Induced angular system; Controllability.
}

\hrule

\section{Introduction}

We consider the bilinear control system given by 
\begin{equation}\label{sp}
\Sigma: \hspace{0.6in}\dot{x}(t)=(A+uB)x(t), \hspace{0.2in} t\in\mathbb{R},\ \ x(t)\in\mathbb{R}^2\smallsetminus\{0\},
\end{equation} where $A, B\in gl(2,\mathbb{R}),$ are $2\times 2$ matrices with real coefficients, and $$u\in\mathcal{U}=\{ u:\mathbb{R}\to U\subset\mathbb{R}:\ u\  \mbox{is locally constant} \} $$ is the set of admissible controls (for this definition and the following results we refer to \cite{CK,E}). We know that a necessary condition for the controllability of the control system $\Sigma$ on $\mathbb{R}^2\smallsetminus\{0\}$ is given by the {\it Lie algebra rank condition} (LARC): the Lie algebra generated by the control system $\Sigma,$ given by $\mathcal{L}_{\Sigma}=\langle A+uB\mid u\in \mathbb{R}  \rangle \subset gl(2,\mathbb{R})$
satisfies, for all $x\in\mathbb{R}^2\smallsetminus\{0\},$
\begin{equation}\label{alglie1}
\dim \mathcal{L}_{\Sigma}(x)=\dim\langle Ax+uBx \mid u\in\mathbb{R} \rangle =2.
\end{equation}

In this paper we will study the controllability problem of the bilinear control system $\Sigma$ on $\mathbb{R}^2\smallsetminus\{0\}.$ The main result, which will be proved in Section \ref{secLARC}, is the following characterization of the Lie algebra rank condition that in many works is assumed as a necessary condition without being explicitly described. 

\begin{teorema}\label{LARC_TEO1}
The bilinear control system $\Sigma$ satisfies the Lie algebra rank condition on $\mathbb{R}^2\smallsetminus \{0\}$ if and only if there exists $\tilde{A}$ and $\tilde{B}$ in $\mathcal{L}_{\Sigma}$ such that 
\begin{equation*}
\mbox{tr}^2\left(\mbox{adj}\left(\tilde{A}\right)\tilde{B}\right)-4\det\left(\mbox{adj}\left(\tilde{A}\right)\tilde{B}\right) < 0,
\end{equation*} where $\mbox{adj}\left(\tilde{A}\right)$ is the classical adjoint matrix of $\tilde{A}.$
\end{teorema} 

An other necessary condition for the controllability of the system $\Sigma$ on $\mathbb{R}^2\smallsetminus\{0\}$ is given by the controllability of the induced {\it angular system} $\mathbb{P}\Sigma$ on the real projective space $\mathbb{P}^1$ defined as the projection of $\Sigma$ on  $\mathbb{P}^ 1 ,$ {\it i.e.,}
\begin{equation}\label{sp0}
\mathbb{P}\Sigma:\hspace{0.6in}\dot{s}(t)=h(A,s(t))+u h(B,s(t)), \ s\in\mathbb{P}^1,
\end{equation} where, for example, $h(A,s)=(A-s^\top AsI_2)s,$ with $I_2$ the $2\times 2$ identity matrix and $u\in \mathcal{U};$ see \cite[Corollary 12.2.6]{CK}. Using a different technique than the one given in \cite{C}, we will show in Theorem \ref{controlproy} that the controllability of $\mathbb{P}\Sigma$ on $\mathbb{P}^1$ is equivalent to the existence of a control $u\in\mathcal{U}$ such that the matrix $A+uB$ has a complex eigenvalue (see Observations \ref{delta}-\ref{obs1}). Finally, in Theorems \ref{critcontrol}-\ref{critcontrol1} we will give controllability criteria for the bilinear system $\Sigma$ on $\mathbb{R}^2\smallsetminus\{0\};$ these results improves the criterion presented in \cite[Theorem 6.3]{C} since we explicitly know the Lie algebra rank condition.

We quote the following related results: the relationship between the controllability of the system $\Sigma$ on $\mathbb{R}^2\smallsetminus\{0\}$ and the controllability of the induced angular system $\mathbb{P}\Sigma$ on $\mathbb{ P}^1$ is given as a particular case of Corollary 12.2.6 in \cite{CK}; see also \cite[Theorem 2.4]{C}. In \cite{C}, the first author and his collaborators study the controllability problem of a bilinear control system in the plane, however, they do not make explicit the Lie algebra rank condition, consequently, the controllability criterion they describe is incomplete, this motivates the present paper. 

The outline of the paper is as follows: in Section \ref{secLARC} we will prove Theorem \ref{LARC_TEO1},  the main result of this paper: the Lie algebra rank condition, a necessary condition for the controllability of the system $\Sigma$ on $\mathbb{R}^2\smallsetminus\{0\},$ is characterized in terms of the existence of  a complex eigenvalue of the matrix $\mbox{adj}(\tilde{A})\tilde{B},$ where $\tilde{A}$ and $\tilde{B}$ belong to $\mathcal{L}_{\Sigma}.$ In Section \ref{conProy} we will study the controllability of the induced angular system $\mathbb{P}\Sigma$ on $\mathbb{P}^1;$ we recover the result obtained in \cite[Theorem 3.3]{C}: suppose that $\Sigma$ satisfies LARC on $\mathbb{R}^2\smallsetminus\{0\},$ then $\mathbb{P}\Sigma$ is controllable on $\mathbb{P}^1$ if and only if there exists $u\in\mathcal{U}$ such that $A+uB$ has a complex eigenvalue; in contrast to the proof given in \cite{C}, we will use an elementary technique. Finally, in Section \ref{secESP} we describe a controllability criterion for the bilinear system $\Sigma$ over $\mathbb{R}^2\smallsetminus\{0\};$ we improve the result given in \cite[Theorem 6,3 (b)]{C} since we make explicit the Lie algebra rank condition of the system.

\section{Lie algebra rank condition: proof of Theorem \ref{LARC_TEO1}}\label{secLARC}

\begin{lema}\label{lema2} For all $x\in \mathbb{R}^2\smallsetminus\{0\},$ $Ax$ and $Bx$ are linearly independent if and only if 
\begin{equation*}
\mbox{tr}^2(\mbox{adj}(A)B)-4\det(\mbox{adj}(A)B) < 0,
\end{equation*} {\it i.e.,} $\mbox{adj}(A)B$ has a complex eigenvalue, where $\mbox{adj}(A)$ is the classical adjoint matrix of $A.$
\end{lema}
\begin{proof}
We have that $Ax$ and $Bx$ are linearly independent, for all $x\in \mathbb{R}^2\smallsetminus\{0\},$ if and only if
\begin{equation}\label{sp1}
\det\begin{pmatrix}
Ax &  Bx  \end{pmatrix}\neq 0
\end{equation} for all $x\in\mathbb{R}^2 \smallsetminus \{0\}.$ Writing $A=\begin{pmatrix}
A_1 & A_2
\end{pmatrix}$ and $B=\begin{pmatrix}
B_1 & B_2
\end{pmatrix},$ with $A_1$ and $A_2$ (resp. $B_1$ and $B_2$) are the column vectors of $A$ (resp. of $B$), for all $x=(x_1, x_2)^\top\in\mathbb{R}^2\smallsetminus\{0\}$ we have
\begin{equation*}
Ax=\begin{pmatrix}
x_1A_1 + x_2A_2
\end{pmatrix}  \hspace{0.3in}\mbox{and}\hspace{0.3in} Bx=\begin{pmatrix}
x_1B_1 + x_2B_2
\end{pmatrix};
\end{equation*} 
using the bilinearity of the determinant we get
\begin{align*}
\det\begin{pmatrix}
Ax &  Bx  \end{pmatrix}
= x_1^2\det\begin{pmatrix}
A_1 &  B_1 \end{pmatrix}+x_1x_2 \left[ \det\begin{pmatrix}
A_1 &  B_2  \end{pmatrix} + \det\begin{pmatrix}
A_2 &  B_1  \end{pmatrix} \right]+x_2^2 \det\begin{pmatrix}
A_2 &  B_2  \end{pmatrix},
\end{align*}
thus, using \eqref{sp1} we have: $Ax$ and $Bx$ are linearly independent, for all $x\in \mathbb{R}^2\smallsetminus\{0\},$ if and only if 
\begin{equation}\label{ecua2}
\left[ \det\begin{pmatrix}
A_1 &  B_2  \end{pmatrix} + \det\begin{pmatrix}
A_2 &  B_1  \end{pmatrix} \right]^2-4\det\begin{pmatrix}
A_1 &  B_1  \end{pmatrix}\det\begin{pmatrix}
A_2 &  B_2  \end{pmatrix} < 0,
\end{equation} {\it i.e.,} the homogeneous equation of second degree of two variables $\det\begin{pmatrix}
Ax &  Bx  \end{pmatrix}=0$  has no real roots. In what follows, we will characterize the inequality \eqref{ecua2}.
Using the Cayley-Hamilton Theorem \cite[Theorem 2.2]{B}: $A^2=\mbox{tr}(A)A-(\det A) I_2,$ for all $A\in gl(2,\mathbb{R}),$ we obtain the formula
\begin{equation}\label{detf1}
\det(A+rB)=\det A + \left[  \mbox{tr}(A)\mbox{tr}(B)- \mbox{tr}(AB) \right]r + (\det B) r^2,
\end{equation} for all $A, B\in gl(2,\mathbb{R})$ and $r\in\R,$ see \cite[Lemma 2.7]{B}.
On the other hand, if $A \in gl(2,\mathbb{R})$ we can write
$\mbox{adj}(A)=\mbox{tr}(A)I_2-A,$ where $\mbox{adj}(A)$ is the classical adjoint matrix of $A,$ thus  
\begin{equation}\label{foradj}
\mbox{tr}(\mbox{adj}(A)B)
 = \mbox{tr}\left( \mbox{tr}(A)B-AB \right) = \mbox{tr}(A)\mbox{tr}(B)-\mbox{tr}(AB),
\end{equation} for all $A, B \in gl(2,\mathbb{R}).$ Moreover, we note that
\begin{align*}
\det(A +  B ) 
&= \det\begin{pmatrix}
A_1+B_1 & A_2+B_2  \end{pmatrix} 
\\ &= \det\begin{pmatrix}
A_1 &  A_2  \end{pmatrix} + \det\begin{pmatrix}
A_1 &  B_2  \end{pmatrix} + \det\begin{pmatrix}
B_1 &  A_2  \end{pmatrix} + \det\begin{pmatrix}
B_1 &  B_2  \end{pmatrix} 
\\ &= \det A  + \det\begin{pmatrix}
A_1 &  B_2  \end{pmatrix} - \det\begin{pmatrix}
A_2 &  B_1  \end{pmatrix} + \det
B;
\end{align*} using \eqref{detf1} and \eqref{foradj} we get
\begin{align}\label{larc1}
\det\begin{pmatrix}
A_1 &  B_2  \end{pmatrix} - \det\begin{pmatrix}
A_2 &  B_1  \end{pmatrix} 
&= \mbox{tr}(\mbox{adj}(A)B);
\end{align}
on the other hand, for all $A, B \in gl(2,\mathbb{R}),$ by a straightforward computation in coordinates, we get
\begin{equation}\label{larc2}
\det(A)\det(B)=\det\begin{pmatrix}
A_1 &  B_1  \end{pmatrix}\det\begin{pmatrix}
A_2 &  B_2  \end{pmatrix} - \det\begin{pmatrix}
A_1 &  B_2  \end{pmatrix}\det\begin{pmatrix}
A_2 &  B_1  \end{pmatrix}. 
\end{equation}
Finally, using \eqref{larc1} and \eqref{larc2}, we directly obtain
\begin{align*}
& \left[ \det\begin{pmatrix}
A_1 &  B_2  \end{pmatrix} + \det\begin{pmatrix}
A_2 &  B_1  \end{pmatrix} \right]^2 -4\det\begin{pmatrix}
A_1 &  B_1  \end{pmatrix}\det\begin{pmatrix}
A_2 &  B_2  \end{pmatrix} = \\
& =  \left[ \det\begin{pmatrix}
A_1 &  B_2  \end{pmatrix} - \det\begin{pmatrix}
A_2 &  B_1  \end{pmatrix} \right]^2 - 4 \left[ \det\begin{pmatrix}
A_1 &  B_1  \end{pmatrix}\det\begin{pmatrix}
A_2 &  B_2  \end{pmatrix} - \det\begin{pmatrix}
A_1 &  B_2  \end{pmatrix}\det\begin{pmatrix}
A_2 &  B_1  \end{pmatrix} \right] \\
&= \hspace{0.1in} \mbox{tr}^2(\mbox{adj}(A)B)-4\det(A)\det(B),
\end{align*} since $\det(\mbox{adj}(A))=\det(A),$  the lemma is proved. 
\end{proof}


\begin{proof}[Proof of Theorem \ref{LARC_TEO1}]
The linear space $\mathcal{L}_{\Sigma}(x)=\langle Ax+uBx\mid u\in \mathbb{R}  \rangle \subset T_x\R^2$ is generated by the vectors of the form $Z_u(x)=(A+uB)x \in \mathbb{R}^2,$ with $u\in\mathbb{R};$ taking $u=0,$ we have a direction given by $Ax;$ moreover, for all $u\neq v$ in $\mathbb{R},$ we have the direction $Z_u(x)-Z_v(x)=(u-v)Bx.$ On the other hand,  for all $u$ and  $v$ in $\mathbb{R},$ with $u\neq v,$ the Lie bracket of the vector fields $Z_u=A+uB$ and $Z_v=A+vB$ in $\mathcal{L}_{\Sigma}$ satisfy 
\begin{align*}
[Z_u,Z_v] = [A+uB,A+vB] 
=(v-u)[A,B],
\end{align*}
{\it i.e.,} the Lie bracket of $Z_u$ and $Z_v,$ at $x,$ is contained in the direction generated by the vector $[A,B]x.$ Thus, all linear combinations of $Z_u(x),$ $u\in\mathbb{R},$ and Lie brackets of these are generated by the set $$\{ Ax,\ Bx, \ [A,B]x,\ [A,[A,B]]x,\ [B,[A,B]]x, \ldots \}.$$ 
Therefore, $\dim\mathcal{L}_{\Sigma}(x)=2,$ for all $x\in\R^{2}\smallsetminus\{0\},$  if and only if there exists $\tilde{A}$ and $\tilde{B}$ in $\mathcal{L}_{\Sigma}$ such that, the vectors $\tilde{A}x$ and $\tilde{B}x$ in $\mathcal{L}_{\Sigma}(x)$ are linearly independent, for all $x\in\R^{2}\smallsetminus\{0\}.$ Using Lemma \ref{lema2} we obtain the proof of Theorem \ref{LARC_TEO1}.
\end{proof}

\begin{obs}\label{obs0} In Theorem \ref{LARC_TEO1}, if  $\tilde{A}=A$ and $\tilde{B}=[A,B],$ since $\mbox{tr}\left(\mbox{adj}(A)[A,B]\right)=0,$
we have $Ax$ and $[A,B]x$ are linearly independent, for all $x\in\mathbb{R}^2\smallsetminus\{0\},$ if and only if
\begin{equation*}
-4\det(A)\det[A,B]=\mbox{tr}^2\left(\mbox{adj}(A)[A,B]\right)-4\det\left(\mbox{adj}(A)[A,B]\right) < 0.
\end{equation*} Therefore, the system $\Sigma$ satisfies the Lie algebra rank condition on $\mathbb{R}^2\smallsetminus\{0\}$ if any of the following conditions holds
\begin{enumerate}
\item $\mbox{tr}^2(\mbox{adj}(A)B)-4\det(\mbox{adj}(A)B) < 0,$ or 
\item $\det(A)\det[A,B]>0$ or $\det(B)\det[A,B]>0.$
\end{enumerate}
\end{obs}

\begin{ejemplo} Let us consider the bilinear control system on $\mathbb{R}^2\smallsetminus \{0\}$ given by
\begin{equation*}
\dot{x}=Ax+uBx=\begin{pmatrix}
-1 & 1 \\ 0 & 1
\end{pmatrix}x + u \begin{pmatrix}
0 & 1 \\ -1 & 0
\end{pmatrix}x. 
\end{equation*} From a direct calculation, since $\mbox{adj}(A)B=\begin{pmatrix}
1 & 1 \\ 1 & 0
\end{pmatrix},$ we have
$$\mbox{tr}^2(\mbox{adj}(A)B)-4\det(\mbox{adj}(A)B)=5 > 0;$$ on the other hand, since $\det(A)\det[A,B]=(-1)(-5)=5 > 0,$ according to Theorem \ref{LARC_TEO1} and Observation \ref{obs0}, the system satisfies the Lie algebra rank condition.
\end{ejemplo}

\begin{ejemplo} Let us consider the bilinear control system on $\mathbb{R}^2\smallsetminus \{0\}$ given by
\begin{equation*}
\dot{x}=Ax+uBx=\begin{pmatrix}
1 & 0 \\ 0 & 0
\end{pmatrix}x + u \begin{pmatrix}
0 & 1 \\ -1 & 0
\end{pmatrix}x. 
\end{equation*} 
From a direct calculation we have $\mbox{tr}^2(\mbox{adj}(A)B)-4\det(\mbox{adj}(A)B)=0$ and
\begin{list}{}{}
\item $\mbox{tr}^2(\mbox{adj}(A)[A,B])-4\det(\mbox{adj}(A)[A,B])=0;$
\item $\mbox{tr}^2(\mbox{adj}(B)[A,B])-4\det(\mbox{adj}(B)[A,B])=4 > 0;$ 
\item $\mbox{tr}^2(\mbox{adj}(A)[A,[A,B]])-4\det(\mbox{adj}(A)[A,[A,B]])=0;$ 
\item $\mbox{tr}^2(\mbox{adj}(B)[A,[A,B]])-4\det(\mbox{adj}(B)[A,[A,B]])=0;$ 
\item $\mbox{tr}^2(\mbox{adj}(A)[B,[A,B]])-4\det(\mbox{adj}(A)[B,[A,B]])=4> 0;$ 
\item $\mbox{tr}^2(\mbox{adj}(B)[B,[A,B]])-4\det(\mbox{adj}(B)[B,[A,B]])=16> 0;$
\item $\mbox{tr}^2(\mbox{adj}([A,B])[A,[A,B]])-4\det(\mbox{adj}([A,B])[A,[A,B]])=4> 0;$
\item $\mbox{tr}^2(\mbox{adj}([A,B])[B,[A,B]])-4\det(\mbox{adj}([A,B])[B,[A,B]])=-16 < 0.$
\end{list} Therefore, according to Lemma \ref{lema2}, the vectors $[A,B]x$ and $[B,[A,B]]x$ in $\mathcal{L}_{\Sigma}(x)$ are linearly independent, for all $x\in\mathbb{R}^2\smallsetminus \{0\},$ thus, the system satisfies the Lie algebra rank condition.
\end{ejemplo}

\begin{ejemplo} Let us consider the bilinear control system on $\mathbb{R}^2\smallsetminus \{0\}$ given by
\begin{equation*}
\dot{x}=Ax+uBx=\begin{pmatrix}
1 & 2 \\ 0 & 1
\end{pmatrix}x + u \begin{pmatrix}
2 & 3 \\ 0 & 2
\end{pmatrix}x. 
\end{equation*} Note that the Lie bracket satisfies $[A,B]=0.$ Since $\mbox{adj}(A)B=\begin{pmatrix}
2 & -1 \\ 0 & 2
\end{pmatrix},$ we have that
\begin{equation*}
\mbox{tr}^2(\mbox{adj}(A)B)-4\det(\mbox{adj}(A)B)=4^2-4(4)=0,
\end{equation*} thus, according to Theorem \ref{LARC_TEO1}, the system does not satisfies the Lie algebra condition. We can verify this as follows: since
$$\mathcal{L}_{\Sigma}(x)=\langle Ax+uBx \mid u\in\mathbb{R} \rangle=\mbox{span}\{Ax,Bx\}=\mbox{span}\begin{pmatrix}
x_1+2x_2 & 2x_1+3x_2 \\ x_2 & 2x_2
\end{pmatrix},$$ we easily get $\dim \mathcal{L}_{A}\begin{pmatrix}
1 \\ 0
\end{pmatrix}=1,$ as expected.
\end{ejemplo}

\section{Controllability of $\mathbb{P}\Sigma$ on the projective space $\mathbb{P}^1$ }\label{conProy}

We consider the bilinear control system $\Sigma$ on $\mathbb{R}^2\smallsetminus\{0\}$ (as in \eqref{sp}) given in coordinates by
\begin{equation}\label{sp001}
\dot{\begin{pmatrix}
x_1 \\ x_2
\end{pmatrix}}=\left(\begin{pmatrix}
a_1 & a_2 \\ a_3 & a_4 \end{pmatrix}+u\begin{pmatrix}
b_1 & b_2 \\ b_3 & b_4 \end{pmatrix}\right)\begin{pmatrix}
x_1 \\ x_2
\end{pmatrix}.
\end{equation} 
Projecting the system \eqref{sp001} on the projective space $\mathbb{P}^1$ we obtain, as in \eqref{sp0}, the induced angular system $\mathbb{P}\Sigma;$ written in coordinates $(s_1,s_2)$ of $\mathbb{P}^1\subset \mathbb{S}^1\subset\mathbb{R}^2$ we have
\begin{equation}\label{sp01}
\dot{\begin{pmatrix}
s_1 \\ s_2
\end{pmatrix}}=\begin{pmatrix}
(a_1+ub_1)s_1s_2^2+(a_2+ub_2)s_2^3-(a_3+ub_3)s_1^2s_2-(a_4+ub_4)s_1s_2^2 \\
-(a_1+ub_1)s_1^2s_2-(a_2+ub_2)s_1s_2^2+(a_3+ub_3)s_1^3+(a_4+ub_4)s_1^2s_2
\end{pmatrix}.
\end{equation}

\begin{obs} We introduce polar coordinates in \eqref{sp01} by the formulas $s_1=\cos \theta$ and  $s_2=\sin\theta,$ we directly get the unique differential equation
\begin{equation}\label{sp04}
\dot{\theta}=\left[ (a_4-a_1)+u(b_4-b_1) \right]\cos\theta\sin\theta-(a_2+ub_2)\sin^2\theta+(a_3+ub_3)\cos^2\theta.
\end{equation} If we define the expressions
\begin{equation}\label{sp05}
P=(a_2+a_3)+u(b_2+b_3), \hspace{0.1in} Q=(a_4-a_1)+u(b_4-b_1) \hspace{0.1in}\mbox{and}\hspace{0.1in} R=(a_3-a_2)+u(b_3-b_2),
\end{equation} which depend on $u,$ the equation \eqref{sp04} remains as 
\begin{equation}\label{sp06}
\dot{\theta}=\frac{1}{2} \left[ P\cos(2\theta)+Q\sin(2\theta)+R \right].
\end{equation}
The controllability problem of the induced angular system  $\mathbb{P}\Sigma$ on $\mathbb{P}^1,$ in terms of $\theta,$ is equivalent to the property that, for some control $u,$ there is a solution of \eqref{sp06}, which has an image diffeomorphic to $(-\pi/2,\pi/2].$ 
\end{obs}

In the following we will integrate the equation
\eqref{sp06}. We have
\begin{equation*}\label{sp08}
\int_{t_0}^{t} \frac{2\dot{\theta}}{P\cos(2\theta)+Q\sin(2\theta)+R}dt=t-t_0;
\end{equation*} 
using the function $v=\tan\theta,$  this relation remains as
\begin{equation}\label{sp09}
\int_{v_0}^{v} \dfrac{2dv}{(R-P)v^2+2Qv+(R+P)}=t-t_0.
\end{equation} 
\begin{obs}\label{obscontrol} The controllability problem of the angular system  $\mathbb{P}\Sigma$ on $\mathbb{P}^1,$ in terms of $v,$ is equivalent to the property that, for some control $u,$ there is a solution of \eqref{sp09}, which has an image diffeomorphic to $(-\infty,+\infty)$ for $\theta(t)\neq\pi/2$ and $\lim_{\theta(t)\to\pi/2} v(t)=+\infty.$
\end{obs}

\noindent {\bf Case 1.} We suppose that $R-P=0.$ We have two possibilities:

\noindent {\bf 1.1.} $Q\neq 0:$ the relation \eqref{sp09} remains as
\begin{equation*}
t-t_0=\int_{v_0}^{v} \dfrac{2dv}{2Qv+(R+P)}=\int_{v_0}^{v} \dfrac{dv}{Qv+P}=\dfrac{1}{Q}\ln \left( Qv+P \right) \Big|_{v_0}^v,
\end{equation*} therefore
\begin{equation*}
v(t)=\left(\frac{Qv_0+P}{Q}\right) e^{(t-t_0)Q}-\frac{P}{Q};
\end{equation*} 
in this case, the induced angular system $\mathbb{P}\Sigma$ does not controllable on $\mathbb{P}^1$ since the image of $v$ is of the form $(a,+\infty)$ or $(-\infty,a),$ for some $a\in\mathbb{R}.$

\noindent {\bf 1.2.} $Q=0:$ we have two cases 
\begin{enumerate}
\item $P\neq 0:$ the relation \eqref{sp09} remains as
\begin{equation*}
t-t_0=\int_{v_0}^{v} \dfrac{2dv}{(R+P)}=\int_{v_0}^{v} \dfrac{dv}{P}=\dfrac{1}{P}(v-v_0).
\end{equation*} 
We have in this case
$v(t)=P(t-t_0)+v_0,$
therefore, $\mathbb{P}\Sigma$ could be controllable on $\mathbb{P}^1;$ however, since
$ R-P=
-2(a_2+ub_2)=0,$ for all $u\in\mathcal{U},$ we have $a_2=b_2=0;$ on the other hand, since $Q=0$ for all $u,$ we obtain $a_1=a_4$ and $b_1=b_4,$ thus 
\begin{equation*}
A=\begin{pmatrix}
a_1 & 0 \\ a_3 & a_1 
\end{pmatrix} 
\hspace{0.3in}\mbox{and}\hspace{0.3in}
B=\begin{pmatrix}
b_1 & 0 \\ b_3 & b_1 
\end{pmatrix}.
\end{equation*} 
We have that $[A,B]=0$ and
\begin{equation*}
\mbox{tr}^2(\mbox{adj}(A)B)-4\det(\mbox{adj}(A)B)=(2a_1b_1)^2-4(a_1b_1)(a_1b_1)=0,
\end{equation*} {\it i.e.,} the system $\Sigma$ does not satisfy LARC on $\mathbb{R}^2\smallsetminus\{0\}.$

\item $P=0:$ since $R=0,$ using \eqref{sp06} we have that $\theta$ is a constant, the angular system does not controllable on $\mathbb{P}^1.$
\end{enumerate}

\noindent {\bf Case 2.} We suppose that $R-P\neq 0.$ In this case, we can write
\begin{equation*}
(R-P)v^2+2Qv+(R+P)=(R-P)\left[ \left( v+\frac{Q}{R-P}\right)^2-\frac{P^2+Q^2-R^2}{(R-P)^2} \right].
\end{equation*}
We define the expression 
\begin{equation}\label{def_delta}
\Delta(u):=P^2+Q^2-R^2, \hspace{0.2in} u\in\mathbb{R}.
\end{equation}

\noindent {\bf 2.1.} $\Delta(u)=P^2+Q^2-R^2=0:$ the relation \eqref{sp09} remains as
\begin{equation*}
t-t_0=\int_{v_0}^{v} \dfrac{2dv}{(R-P)v^2+2Qv+(R+P)}= \dfrac{2}{(P-R)v-Q}\Big|_{v_0}^v;
\end{equation*} writing $C_1=\frac{2}{(P-R)v_0-Q},$ we have  
\begin{equation*}
v(t)=\left(\frac{1}{P-R}\right)\frac{2}{t-t_0+C_1}+ \frac{Q}{P-R};
\end{equation*} 
in this case, the angular system does not controllable on $\mathbb{P}^1$ since the image of $v$ has the form $(-\infty,a)\cup(a,+\infty)$ for some $a\in\mathbb{R}.$

\noindent {\bf 2.2.} $\Delta(u)=P^2+Q^2-R^2>0:$ the relation \eqref{sp09} remains as
\begin{equation*}
t-t_0=\int_{v_0}^{v} \dfrac{2dv}{(R-P)v^2+2Qv+(R+P)}= \dfrac{1}{\sqrt{\Delta(u)}}\ln\left( \dfrac{(R-P)v+Q-\sqrt{\Delta}(u)}{(R-P)v+Q+\sqrt{\Delta(u)}} \right)\Big|_{v_0}^v.
\end{equation*} Writing $C_3=\ln\left( \frac{(R-P)v+Q-\sqrt{\Delta(u)}}{(R-P)v+Q+\sqrt{\Delta(u)}} \right),$ we obtain
\begin{equation*} \displaystyle
v(t)=\left( \frac{\sqrt{\Delta(u)}}{R-P} \right) \frac{e^{C_3}+e^{(t-t_0)\sqrt{\Delta(u)}}}{e^{C_3}-e^{(t-t_0)\sqrt{\Delta(u)}}}-\frac{Q}{R-P};
\end{equation*}
as in the previous case, the system does not controllable on $\mathbb{P}^1$ since the image of $v$ has the form $(-\infty,a)\cup(a,\infty)$ for some $a\in\mathbb{R}.$

\noindent {\bf 2.3.} $\Delta(u)=P^2+Q^2-R^2<0:$ in this case, \eqref{sp09} remains as
\begin{equation*}
t-t_0=\int_{v_0}^{v} \dfrac{2dv}{(R-P)v^2+2Qv+(R+P)}= \dfrac{2}{\sqrt{-\Delta(u)}}\arctan \left( \dfrac{(R-P)v+Q}{\sqrt{-\Delta(u)}} \right)\Big|_{v_0}^v.
\end{equation*} Writing $C_2=\arctan \left( \frac{(R-P)v_0+Q}{\sqrt{-\Delta(u)}} \right),$ we get
\begin{equation*}
v(t)=\frac{\sqrt{-\Delta(u)}}{R-P}\tan \left( \sqrt{-\Delta(u)} \frac{(t-t_0)}{2}+C_2 \right)-\frac{Q}{R-P};
\end{equation*}
in this case, the angular system $\mathbb{P}\Sigma$ is controllable on $\mathbb{P}^1.$

\begin{obs} We note that, the condition $\Delta(u)<0$ for some $u\in\mathcal{U},$ implies $R-P\neq 0.$ In fact, if $R=P,$ we obtain $\Delta(u)=P^2+Q^2-R^2=Q^2\geq 0$ for all $u.$ 
\end{obs}

We obtain the following theorem.

\begin{teorema}\label{controlproy} We consider the bilinear control system $\Sigma,$ given in \eqref{sp}, on $\mathbb{R}^2\smallsetminus\{0\}$ assuming that it satisfies LARC. Then, the induced angular system $\mathbb{P}\Sigma,$ given in \eqref{sp0}, is controllable on the projective space $\mathbb{P}^1$ if and only if there exists $u\in\mathbb{R}$ such that $$\Delta(u)=P^2+Q^2-R^2<0.$$ 
\end{teorema}
\begin{proof} According to Observation \ref{obscontrol}, the controllability of $\mathbb{P}\Sigma$ on $\mathbb{P}^1$ depends of the solutions of \eqref{sp09}; the previous analysis (Cases 1 and 2) allows us to deduce that the condition required for the solutions of \eqref{sp09} is satisfied only if $\Delta(u)<0,$ for some $u\in\mathbb{R}.$
\end{proof}

The following observation relates the previous theorem with the criterion of controllability of the induced angular system $\Sigma\mathbb{P}$ on $\mathbb{P}^1$ given in \cite[Theorem 3.3]{C}.

\begin{obs}[Eigenvalues of $A+uB$]\label{delta} We recall the definition of $P, Q$ and $R,$ given in \eqref{sp05}. We consider $S:=(a_4+a_1)+u(b_4+b_1);$ clearly we have
\begin{equation*}
A+uB=\begin{pmatrix}
a_1+ub_1 & a_2+ub_2 \\
a_3+ub_3 & a_4+ub_4
\end{pmatrix}=\dfrac{1}{2}\begin{pmatrix}
S-Q & P-R \\
P+R & S+Q
\end{pmatrix},
\end{equation*} thus we obtain
\begin{equation*}
\mbox{tr}(A+uB)=S \hspace{0.3in}\mbox{and}\hspace{0.3in} \det(A+uB)=\frac{1}{4} \left[ S^2+\left( R^2-P^2-Q^2 \right) \right];
\end{equation*} therefore
\begin{equation*}
\mbox{tr}^2(A+uB)-4\det(A+uB)=P^2+Q^2-R^2 = \Delta(u);
\end{equation*} 
{\it i.e.,} there exists $u\in\mathbb{R}$ such that $\Delta(u)<0$ if and only if there exists $u\in\mathbb{R}$ such that the matrix $A+uB$ has a complex eigenvalue.
\end{obs}

\begin{obs}\label{obs1}
According the previous observation and Theorem \ref{controlproy}, the controllability problem of the induced angular system $\mathbb{P}\Sigma$ on the projective space $\mathbb{P}^1$ is characterized as follows: with the same hypothesis of Theorem \ref{controlproy}, the induced angular system $\mathbb{P}\Sigma$ is controllable on $\mathbb{P}^1$ if and only if there exists $u\in\mathbb{R}$ such that $A+uB$ has a complex eigenvalue. The same characterization was proved in \cite[Theorem 3.3]{C} but using a different technique; in our proof, the technique is elementary.
\end{obs}

\begin{ejemplo}\label{ej1} The bilinear control system, on $\mathbb{R}^2\smallsetminus\{0\},$ given by
\begin{equation*}
\dot{x}=\begin{pmatrix}
2 & -1 \\ 0 & 1
\end{pmatrix}x + u \begin{pmatrix}
0 & 1 \\ -1 & 0
\end{pmatrix}x, 
\end{equation*} satisfies LARC; in fact, we easily get
\begin{equation}\label{ej3_1}
\mbox{tr}^2(\mbox{adj}(A)B)-4\det(\mbox{adj}(A)B)=\mbox{tr}^2\begin{pmatrix}
-1 & 1 \\ -2 & 0
\end{pmatrix}-4\det\begin{pmatrix}
-1 & 1 \\ -2 & 0
\end{pmatrix}=-7 <0.
\end{equation} Moreover, the induced angular system $\mathbb{P}\Sigma$ is controllable on $\mathbb{P}^1;$ in fact, we have
\begin{align*}
\Delta(u) =\mbox{tr}^2(A+uB)-4\det(A+uB) &=\mbox{tr}^2\begin{pmatrix}
2 & u-1 \\ -u & 1
\end{pmatrix} -4\det\begin{pmatrix}
2 & u-1 \\ -u & 1
\end{pmatrix} \\ 
&=-4u^2+4u+1;
\end{align*} thus $\Delta(u)<0,$ for every constant control $u\in \left(-\infty, \frac{1-\sqrt{2}}{2}\right)\cup \left(\frac{1+\sqrt{2}}{2},+\infty\right).$
\end{ejemplo}

\subsection{Conditions on $A$ and $B$ for $\Delta(u)<0$}

In this section we will describe conditions on $A$ and $B,$ such that $\Delta(u)<0,$ for some $u\in\mathbb{R}.$ 

Using the relation \eqref{detf1} and Observation \ref{delta} we have
\begin{align*}
\Delta(u) &=\mbox{tr}^2(A+uB)-4\det(A+uB) \\
&= \left[ \mbox{tr}^2(B)-4\det B\right]u^2 + 2 \left[ 2\mbox{tr}(AB)-\mbox{tr}(A)\mbox{tr}(B)\right] u+\left[ \mbox{tr}^2(A)-4\det A\right] ;
\end{align*} writing
\begin{align}\label{abg}
\alpha:=\left[ \mbox{tr}(B)^2-4\det(B) \right],\hspace{0.1in} \beta:= 2 \left[ 2\mbox{tr}(AB)-\mbox{tr}(A)\mbox{tr}(B)\right], \hspace{0.1in} \gamma:=\left[ \mbox{tr}^2(A)-4\det A\right],
\end{align} we get the expression
\begin{equation}\label{delta1}
\Delta(u)=\alpha u^2+ \beta u+ \gamma, \hspace{0.3in} u\in \mathbb{R}.
\end{equation} 
We note that, according to \eqref{foradj}, we have
\begin{equation*}
\beta=0\hspace{0.2in} \Leftrightarrow \hspace{0.2in} 2\mbox{tr}(AB)=\mbox{tr}(A)\mbox{tr}(B) \hspace{0.2in} \Leftrightarrow \hspace{0.2in} \mbox{tr}(\mbox{adj}(A)B)=\mbox{tr}(AB).
\end{equation*}
We will determine conditions on the coefficients $\alpha, \beta$ and $\gamma,$ such that $\Delta(u)< 0,$ for some $u\in\mathbb{R}.$ We need the following lemma.

\begin{lema}\label{lemadiscriminante} The discriminant of $\Delta(u)=\alpha u^2+\beta u+\gamma=0,$ with $u\in\mathbb{R},$ is given by
\begin{equation}\label{discriminante}
\beta^2-4\alpha\gamma=-16\det[A,B].
\end{equation}
\end{lema}

\begin{proof} The discriminant of  $\alpha u^2+\beta u+\gamma=0,$ with $u\in\mathcal{U},$  is given by
\begin{align*}
\beta^2-4\alpha\gamma &=4\left[ 2\mbox{tr}(AB)-\mbox{tr}(A)\mbox{tr}(B) \right]^2-4\left[ \mbox{tr}(B)^2-4\det(B) \right]\left[ \mbox{tr}(A)^2-4\det(A) \right] \\
&= 16 \left[ \mbox{tr}^2(AB)-\mbox{tr}(A)\mbox{tr}(B)\mbox{tr}(AB)+\mbox{tr}^2(B)\det(A)+\mbox{tr}^2(A)\det(B)-4\det(AB) \right].
\end{align*} 
Using the relation \eqref{detf1}, with $r=-1,$ we have
\begin{align*}
\det[A,B]
= 2\det(AB)-\mbox{tr}^2(AB)+\mbox{tr}(A^2B^2);
\end{align*}
of Cayley-Hamilton Theorem we get $A^2=\mbox{tr}(A)A-\det(A)I_2$ and $B^2=\mbox{tr}(B)B-\det(B)I_2,$ therefore
\begin{equation*}
\mbox{tr}(A^2B^2) 
= \mbox{tr}(A)\mbox{tr}(B)\mbox{tr}(AB)-\mbox{tr}^2(A)\det(B)-\mbox{tr}^2(B)\det(A) +2\det(AB),
\end{equation*} 
thus, finally we get
\begin{align*}
\det[A,B]
&=-\left[ \mbox{tr}^2(AB)- \mbox{tr}(A)\mbox{tr}(B)\mbox{tr}(AB)+\mbox{tr}^2(A)\det(B)+\mbox{tr}^2(B)\det(A)-4\det(AB) \right],
\end{align*}  as we expected.
\end{proof}

\begin{obs} If $\alpha\neq 0,$ we can write
\begin{equation}\label{expdelta1}
\Delta(u)=\alpha u^2+\beta u+\gamma = \alpha \left( \left(u+\dfrac{\beta}{2\alpha}\right)^2-\dfrac{\beta^2-4\alpha\gamma}{4\alpha^2} \right),
\end{equation}
thus, when considering the constant control $u=-\frac{\beta}{2\alpha},$ the maximum value (resp. minimum value) of $\Delta(u),$ if $\alpha<0$ (resp. if $\alpha>0$), is given by
\begin{equation*}
\Delta\left(-\frac{\beta}{2\alpha}\right)=\frac{4\det[A,B]}{\alpha}.
\end{equation*}
\end{obs}
We considerer the following cases according to the sign of $\det[A,B].$ \\

\noindent{\bf Case A.} Suppose that $\det[A,B]<0:$ this condition means that $\beta^2-4\alpha\gamma>0;$ we have two possibilities 

\noindent {\bf A.1.} $\alpha=0:$ {\it i.e.,} $B$ has a unique real eigenvalue. We necessarily have $\beta\neq 0,$ thus, there exists $u$ such that $\Delta(u)=\beta u+\gamma <0.$
\begin{center}
\includegraphics[scale=0.5]{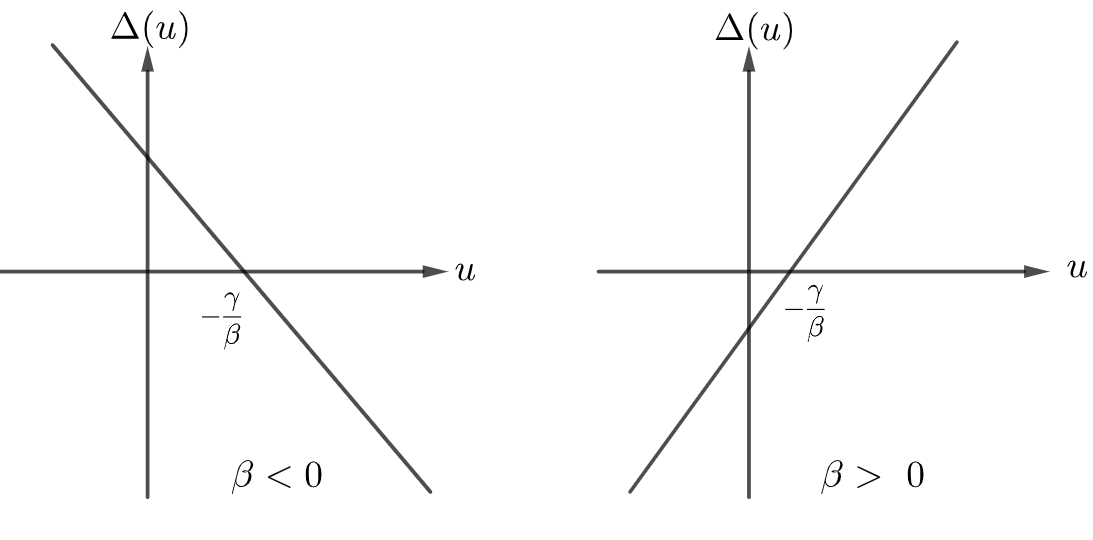}
\end{center}

\noindent {\bf A.2.} $\alpha\neq 0:$ in this case, using \eqref{expdelta1}, we can find a control $u$ such that $\Delta(u)<0.$ For example, if $\alpha>0$ we consider a control such that $\frac{\beta-\sqrt{\beta^2-4\alpha\gamma}}{2\alpha} < u < \frac{\beta+\sqrt{\beta^2-4\alpha\gamma}}{2\alpha}.$
\begin{center}
\includegraphics[scale=0.5]{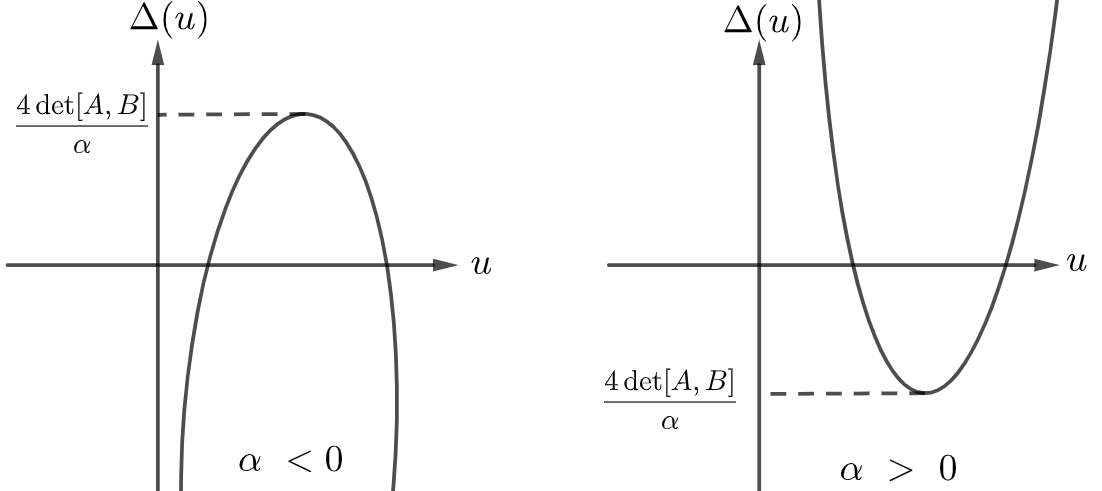}
\end{center}

\noindent{\bf Case B.} Suppose that $\det[A,B]>0:$ in this case we have $\beta^2-4\alpha\gamma<0,$ thus $\alpha\neq 0.$ Writing $\Delta(u)$ as in \eqref{expdelta1}, we have the following possibilities:

\noindent {\bf B.1.} $\alpha<0:$ {\it i.e., $B$ has a complex eigenvalue.} In this case, for all $u$ we have $\Delta(u)<0.$

\noindent {\bf B.2.} $\alpha>0:$ {\it i.e., $B$ has two real different eigenvalues.} For all $u\in\mathbb{R},$ we have $\Delta(u)> 0,$ therefore, the induced angular system $\Sigma\mathbb{P}$ does not controllable on $\mathbb{P}^1.$
\begin{center}
\includegraphics[scale=0.5]{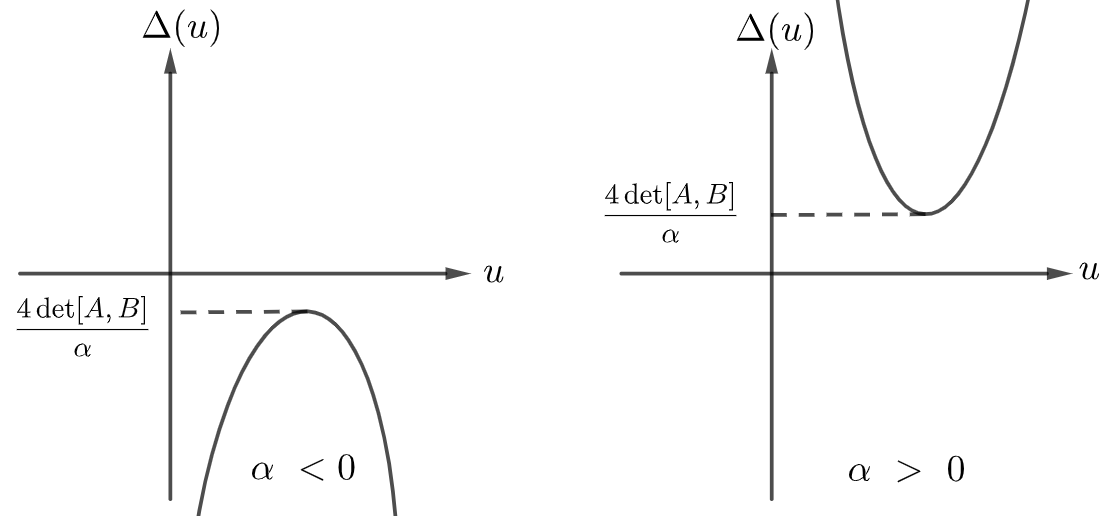}
\end{center}

\noindent{\bf Case C.} Suppose that $\det[A,B]=0:$ this means $\beta^2-4\alpha\gamma=0;$ we have two possibilities  

\noindent {\bf C.1.} $\alpha=0:$ in this case we have $\beta=0,$ thus $\Delta(u)=\gamma$ is a constant. We have $\Delta(u)<0$ for all $u\in\mathbb{R}$ if and only if $\gamma<0,$ {\it i.e.,} if and only if $A$ has a complex eigenvalue.
\begin{center}
\includegraphics[scale=0.5]{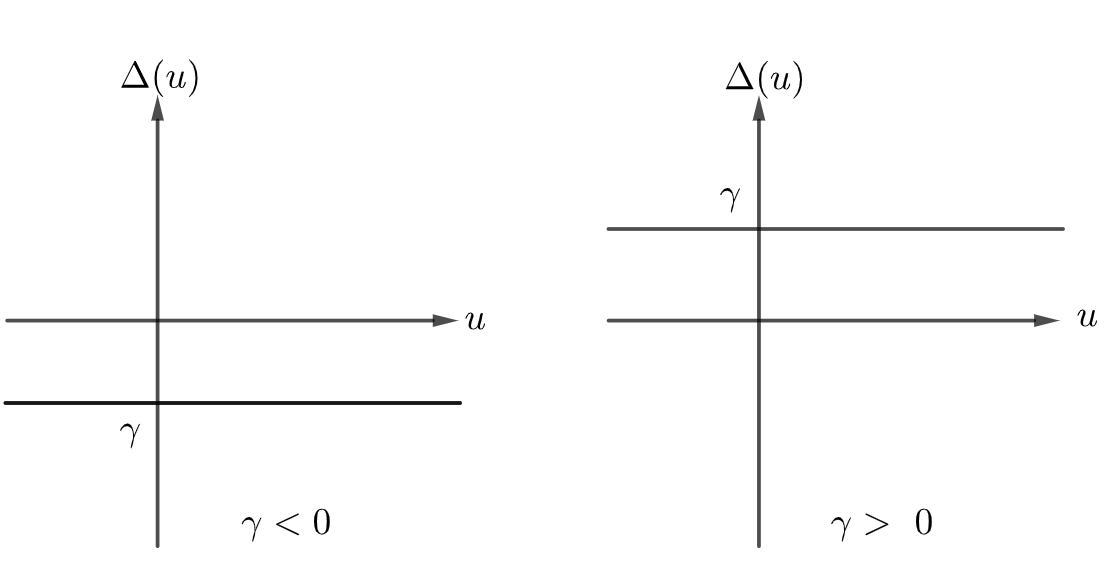}
\end{center}

\noindent {\bf C.2.} $\alpha\neq 0:$ from \eqref{expdelta1}, $\Delta(u)\leq 0,$ for all $u\in\mathbb{R}$ if $\alpha<0,$ {\it i.e.,} if $B$ has a complex eigenvalue. We note that, according to $\beta^2-4\alpha\gamma=0,$ $\alpha<0$ ($B$ has a complex eigenvalue) if and only if $\gamma<0$ ($A$ has a complex eigenvalue). Moreover, if $\alpha>0,$ then $\Delta(u)\geq 0,$ for all $u\in\mathbb{R}.$
\begin{center}
\includegraphics[scale=0.5]{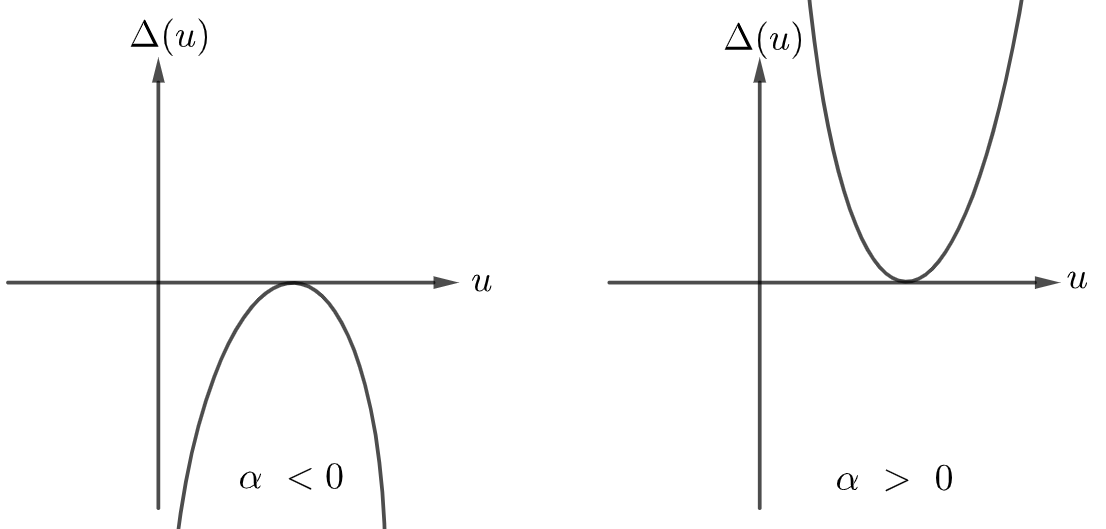}
\end{center}
\begin{proposicion} There exists $u\in\mathbb{R}$ such that $\Delta(u)<0$ if and only if
\begin{enumerate}
\item[\it 1.] $\det[A,B]<0,$ or
\item[\it 2.] $\det[A,B]>0$ and $B$ has a complex eigenvalue, or
\item[\it 3.] $\det[A,B]=0$ and 
\begin{enumerate}
\item $B$ has a unique real eigenvalue, $\mbox{tr}(\mbox{adj}(A)B)=\mbox{tr}(AB)$ and $A$ has a complex eigenvalue, or
\item $B$ or $A$ has a complex eigenvalue.
\end{enumerate}
\end{enumerate}
\end{proposicion}
\begin{proof} Suppose that $\Delta(u)=\alpha u^2+ \beta u +\gamma <0,$ for some $u\in\mathbb{R};$ we will prove that necessarily some condition {\it 1., 2.} or {\it 3.} holds. We consider the following cases: \\
\noindent{Case 1.} $\alpha=0:$ we have the following possibilities \\
\noindent {1.a.} $\beta=0:$ we have $\gamma<0$ since $\Delta(u)<0;$ therefore $\det[A,B]=\frac{-1}{16}(\beta^2-4\alpha\gamma)=0.$\\
\noindent {1.b.} $\beta\neq 0:$ we have that $\det[A,B]=\frac{-1}{16}(\beta^2-4\alpha\gamma)=-\frac{1}{16}\beta^2< 0.$

\noindent{Case 2.} $\alpha<0:$ in this case, every possibility holds according to the sign of $\det[A,B].$ 

\noindent{Case 3.} $\alpha>0:$ writing $\Delta$ as in \eqref{expdelta1}, the unique possibility such that $\Delta(u)<0,$ for some $u\in\mathbb{R},$ is $\beta^2-4\alpha\gamma>0,$ {\it i.e.,} $\det[A,B]<0.$  
\end{proof}

A similar result, but with other technique, was given in \cite[Theorem 4.3]{C}, however, the authors in \cite{C} do not consider the case {\it 3.(a)} of the previous proposition.

\begin{ejemplo}\label{ejn} The bilinear control system, on $\mathbb{R}^2\smallsetminus\{0\},$ given by
\begin{equation*}
\dot{x}=Ax+uBx=\begin{pmatrix}
a & -b \\ b & a
\end{pmatrix}x + u \begin{pmatrix}
c & 0 \\ 0 & c
\end{pmatrix}x, 
\end{equation*} where $bc\neq 0,$ satisfies LARC; in fact, we easily get
\begin{equation*}
\mbox{tr}^2(\mbox{adj}(A)B)-4\det(\mbox{adj}(A)B)=\mbox{tr}^2\begin{pmatrix}
ac & bc \\ -bc & ac
\end{pmatrix}-4\det\begin{pmatrix}
ac & bc \\ -bc & ac
\end{pmatrix}=-4(bc)^2 <0.
\end{equation*} 
Moreover, the induced angular system $\mathbb{P}\Sigma$ is controllable on $\mathbb{P}^1;$ in fact, we have
\begin{align*}
\Delta(u) &=\mbox{tr}^2(A+uB)-4\det(A+uB) \\ &=\mbox{tr}^2\begin{pmatrix}
a+uc & -b \\ b & a+uc
\end{pmatrix} -4\det\begin{pmatrix}
a+uc & -b \\ b & a+uc
\end{pmatrix} \\ 
&=-4b^2
\end{align*} thus $\Delta(u)<0,$ for every constant control $u.$ Finally we note that, this system satisfies the condition {\it 3.(a)} of previous proposition.
\end{ejemplo}

\section{Controllability criterion for  $\Sigma$ on $\mathbb{R}^2\smallsetminus\{0\}$}\label{secESP}

For finish this paper, in this section we consider the Lyapunov spectrum $\Sigma_{L_y}$ of the bilinear control system $\Sigma$ given in \eqref{sp}, {\it i.e.,} we consider the set of Lyapunov exponents
\begin{equation*}
\Sigma_{Ly}=\left\lbrace \lambda(u,x): (u,x)\in\mathcal{U}\times \left( \mathbb{R}^2\smallsetminus\{0\} \right) \right\rbrace,
\end{equation*} where 
$\lambda(u,x)=\limsup_{t\to\infty} \frac{1}{t}\log |\phi(t,x,u)|,$ with $\phi(t,x,u)$ is a solution of $\Sigma.$

We will assume that $\Sigma$ satisfies the Lie algebra rank condition on $\mathbb{R}^2\smallsetminus\{0\}$ and that the induced angular system $\mathbb{P}\Sigma$ is controllable on $\mathbb{P}^1.$ According to \cite[Corollary 12.2.6]{CK}, we need to characterize the condition  $0\in\mbox{int} \Sigma_{L_y}.$

The eigenvalues of $A+uB,$ for $u\in \mathbb{R},$ according to Observation \ref{delta}, are given by
\begin{align*}
\lambda_{1,2}(u) 
= \frac{1}{2} \left( \mbox{tr}(A)+u\  \mbox{tr}(B)\pm \sqrt{\Delta(u)} \right);
\end{align*} 
using the notation
$ \Sigma_{Re}=\{ \mbox{Re}(\lambda_{1,2}(u)): u\in\mathbb{R} \},$ we have $\Sigma_{Re}\subset \Sigma_{L_y}.$ The following lemmas were proved in \cite{C}. 

\begin{lema}\cite[Lemma 6.1]{C}\label{prop1} If the angular system $\mathbb{P}\Sigma,$ given by \eqref{sp0}, is controllable on $\mathbb{P}^1$ and $\mbox{tr}(B)\neq 0,$ then $0\in \mbox{int}(\Sigma_{Re}).$
\end{lema}

\begin{lema}\cite[Lemma 6.2]{C}\label{prop2} If the angular system $\mathbb{P}\Sigma,$ given by \eqref{sp0}, is controllable on $\mathbb{P}^1$ and $\mbox{tr}(B)= 0,$ then $0\in \mbox{int}(\Sigma_{Re})$ if and only if $0<\mbox{tr}^2(AB)-4\det(AB).$
\end{lema}
\begin{ejemplo} We consider the bilinear control system on $\mathbb{R}^2\smallsetminus\{0\}$ given in Example \ref{ej1}:
\begin{equation*}
\dot{x}=\begin{pmatrix}
2 & -1 \\ 0 & 1
\end{pmatrix}x + u \begin{pmatrix}
0 & 1 \\ -1 & 0
\end{pmatrix}x; 
\end{equation*} we recall that, this system satisfies LARC (see \eqref{ej3_1}) and the induced angular systems $\mathbb{P}\Sigma$ is controllable on the projective space $\mathbb{P}^1.$ We note that $\mbox{tr}(B)=0.$ The eigenvalues of $A+uB$ are given by
\begin{align*}
\lambda_{1,2}(u) &=\frac{1}{2} \left( \mbox{tr}(A+uB)\pm \sqrt{\left[ \mbox{tr}^2(A+uB) \right]-4\det(A+uB)} \right) \\
&= \frac{1}{2} \left( 3\pm \sqrt{1+4u-4u^2} \right);
\end{align*} now, we note that $\Delta(u)=1+4u-4u^2\geq 0$ for all $u\in \left[ \frac{1-\sqrt{2}}{2}, \frac{1+\sqrt{2}}{2} \right];$ moreover, the maximum value of  $\Delta(u)$ is given by
\begin{equation*}
\frac{4\det[A,B]}{-4}=-\det\begin{pmatrix}
1 & 1 \\ 1 & -1
\end{pmatrix}=2,
\end{equation*} therefore $\mbox{Re}(\lambda_{1,2}(u))\in \left[ \frac{3-\sqrt{2}}{2},\frac{3+\sqrt{2}}{2} \right],$ thus $0\notin \mbox{int}(\Sigma_{Re}).$ 
\end{ejemplo}

\begin{obs} The conclusion of the previous example coincide with the affirmation of  Lemma \ref{prop2}: in fact, if $\mbox{tr}(B)=0,$ according to \eqref{foradj} we have $\mbox{tr}(\mbox{adj}(A)B)=-\mbox{tr}(AB);$ therefore, from previous example y \eqref{ej3_1}, the Lie algebra rank condition means
\begin{equation*}
\mbox{tr}^2(AB)-4\det(AB)=\mbox{tr}^2(\mbox{adj}(A)B)-4\det(\mbox{adj}(A)B)<0.
\end{equation*}
\end{obs} 

Finally we obtain the followings controllability criteria for the bilinear control system $\Sigma$ on $\mathbb{R}^2\smallsetminus\{0\};$ these results improve the criterion presented in \cite[Theorem 6.3]{C} since we explicitly know the Lie algebra rank condition for $\Sigma.$ According to Theorems \ref{LARC_TEO1} and \ref{controlproy}, Observation \ref{obs1} and Lemma \ref{prop1} we have:
\begin{teorema}\label{critcontrol} The bilinear control system $\Sigma$ is controllable on $\mathbb{R}^2\smallsetminus\{0\}$ if 
\begin{enumerate}
\item $\mbox{tr}^2\left(\mbox{adj}\left(\tilde{A}\right)\tilde{B}\right)-4\det\left(\mbox{adj}\left(\tilde{A}\right)\tilde{B}\right) < 0,$ for some $\tilde{A}$ and $\tilde{B}$ in $\mathcal{L}_{\Sigma},$
\item there exists $u\in\mathcal{U}$ such that $\Delta(u)< 0,$ ({\it i.e.,} $A+uB$ has a complex eigenvalue), and
\item $\mbox{tr}(B)\neq 0.$
\end{enumerate} 
\end{teorema}

\noindent Analogously, from Theorem \ref{controlproy}, Observations \ref{obs0} and \ref{obs1}, and Lemma \ref{prop2}, we have:

\begin{teorema}\label{critcontrol1} The bilinear control system $\Sigma$ is controllable on $\mathbb{R}^2\smallsetminus\{0\}$ if 
\begin{enumerate}
\item $\det(A)\det[A,B]>0$ or $\det(B)\det[A,B]>0,$
\item there exists $u\in\mathcal{U}$ such that $\Delta(u)< 0,$ ({\it i.e.,} $A+uB$ has a complex eigenvalue), and
\item $\mbox{tr}(B)=0.$
\end{enumerate} 
\end{teorema}

\paragraph{\bf Acknowledgements.} The authors thanks the members of the Control Dynamics Project for the different talks and comments on the results of this work.

\Addresses

\end{document}